\documentclass[a4paper]{article}
\usepackage{hyperref}
\usepackage{graphicx}
\usepackage{mathrsfs}
\usepackage{amsmath}
\usepackage{enumerate} 
\usepackage{amsxtra,amssymb,latexsym, amscd,amsthm}
\usepackage{indentfirst}
\usepackage{color}
\usepackage[utf8]{inputenc}
\usepackage[mathscr]{eucal}
\usepackage{amsfonts}
\usepackage{graphics}
\usepackage{multirow}
\usepackage{array}
\usepackage{subfigure}
\usepackage{cite}
\usepackage{wrapfig}

\newcommand{\footremember}[2]{%
    \footnote{#2}
    \newcounter{#1}
    \setcounter{#1}{\value{footnote}}%
}
\newcommand{\footrecall}[1]{%
    \footnotemark[\value{#1}]%
} 

\def\R{{\mathbb R}}
\def\Q{{\mathbb Q}}
\def\N{{\mathbb N}}
\def\S{{\mathbb S}}

\DeclareMathOperator{\Sper}{Sper}
\DeclareMathOperator{\Hom}{Hom}

\newtheorem{lemma}{\bf Lemma}[section]

\newtheorem{assumption}{\bf Assumption}[section]
\newtheorem{example}{\bf Example}[section]
\newtheorem{theorem}{\bf Theorem}[section]

\newtheorem{corollary}{\bf Corollary}[section]

\newtheorem{remark}{\bf Remark}[section]
\providecommand{\keywords}[1]
{
  \small	
  \textbf{\textbf{Keywords:}} #1
}
\begin{document}
\small
\definecolor{qqzzff}{rgb}{0,0.6,1}
\definecolor{ududff}{rgb}{0.30196078431372547,0.30196078431372547,1}
\definecolor{xdxdff}{rgb}{0.49019607843137253,0.49019607843137253,1}
\definecolor{ffzzqq}{rgb}{1,0.6,0}
\definecolor{qqzzqq}{rgb}{0,0.6,0}
\definecolor{ffqqqq}{rgb}{1,0,0}
\definecolor{uuuuuu}{rgb}{0.26666666666666666,0.26666666666666666,0.26666666666666666}
\newcommand{\vi}[1]{\textcolor{blue}{#1}}
\newif\ifcomment
\commentfalse
\commenttrue
\newcommand{\comment}[3]{%
\ifcomment%
	{\color{#1}\bfseries\sffamily#3%
	}%
	\marginpar{\textcolor{#1}{\hspace{3em}\bfseries\sffamily #2}}%
	\else%
	\fi%
}
\newcommand{\victor}[1]{
	\comment{blue}{V}{#1}
}
\title{A sparse version of Reznick's Positivstellensatz}
\author{%
Ngoc Hoang Anh Mai\footremember{1}{CNRS; LAAS; 7 avenue du Colonel Roche, F-31400 Toulouse; France.}, %
  Victor Magron\footrecall{1}  , %
   Jean-Bernard Lasserre\footrecall{1} \footremember{2}{Universit\'e de Toulouse; LAAS; F-31400 Toulouse, France.} %
  }
\maketitle
\begin{abstract}
If $f$ is a positive definite form, Reznick's Positivstellensatz [Mathematische Zeitschrift. 220 (1995), pp. 75--97] states that  there exists $k\in\N$ such that ${\| x \|^{2k}_2}f$ is a sum of squares of polynomials. 
Assuming that $f$ can be written as a sum of forms $\sum_{l=1}^p f_l$, where each $f_l$ depends on a subset of the initial variables, and assuming that these subsets satisfy the so-called {\em running intersection property}, we provide a sparse version of Reznick's Positivstellensatz.
Namely, there exists $k \in \N$ such that
$f=\sum_{l = 1}^p {{\sigma_l}/{H_l^{k}}}$, where $\sigma_l$ is a sum of squares of polynomials, $H_l$ is a uniform polynomial denominator, and both polynomials $\sigma_l,H_l$ involve the same variables as $f_l$, for each $l=1,\dots,p$.
In other words, the sparsity pattern of $f$ is also reflected in this sparse version of  Reznick's certificate of positivity. We next use this result to also obtain positivity certificates for (i) polynomials nonnegative on the whole space  and (ii) polynomials nonnegative on a (possibly non-compact) basic semialgebraic set, assuming that the input data satisfy the running intersection property. Both are sparse versions of a positivity certificate due to Putinar and Vasilescu.
\end{abstract}
\tableofcontents
\section{Introduction and overview}

Before the 1990s, representations of positive polynomials, also known as {\em Positivstellens\"atze}, have been discovered within a purely theoretical branch of real algebraic geometry. 
More recently, such Positivstellens\"atze have become a powerful tool in polynomial optimization and control. 
%
\paragraph{Positivstellens\"atze and polynomial optimization.} 
%
With $x = (x_1,\dots,x_n)$, let $\R[x]$ stands for the ring of real polynomials and let $\Sigma[x]\subset\R[x]$ be the subset of {\em sums of squares} (SOS) of polynomials.
Let us note $\R[x]_d$ and $\Sigma[x]_d$ the respective restrictions of these two sets to polynomials of degree at most $d$ and $2d$. 

SOS decompositions of nonnegative polynomials have a distinguishing feature with important practical implications: 
Indeed they are \emph{tractable} and can be determined by solving a {\em semidefinite program}\footnote{Semidefinite programming (SDP) is an important class of convex conic optimization problems that can be solved efficiently, up to arbitrary precision, fixed in advance; the interested read is referred to e.g. \cite[Chapter 4]{ben2001lectures}.}. 
Namely, writing a polynomial $f \in \R[x]_{2d}$  as an SOS boils down \cite{parrilo2000structured} to computing the entries of a symmetric (Gram) matrix $G$ with only nonnegative eigenvalues (denoted by ``$G\succeq 0$'') such that $f=v_d^TGv_d$, with $v_d$ being the vector of all monomials of degree at most $d$.

Given $f, g_1,\dots,g_m \in \R[x]$, and the basic semialgebraic set $S(g):=\{ x\in\R^n:\, g_j(x)\ge 0\,,\,j=1,\dots,m\}$, with $g:=\{g_1,\dots,g_m\}$,
\emph{polynomial optimization} is concerned with computing $f^\star:=\inf\{f(x):\,x\in S(g)\}$.
A basic idea is to rather consider $f^\star=\sup\{\lambda \in \R:\,f-\lambda > 0\text{ on } S(g)\}$ 
and replace the difficult constraint ``$f-\lambda > 0$ on $S(g)$" with a more tractable
SOS-based decomposition of $f-\lambda$, thanks to various certificates of positivity on $S(g)$.
For instance, if $S(g)$ is compact and satisfies the so-called {Archimedean assumption\footnotemark}, Putinar's Positivstellensatz \cite{putinar1993positive} provides the decomposition $f-\lambda=\sigma_0+\sum_{j=1}^m\sigma_jg_j$, with $\sigma_j \in \Sigma[x]$.
\addtocounter{footnote}{-1}
Then one obtains the monotone non-decreasing sequence $(\rho_k)_{k\in\N}$ of lower bounds on $f^\star$ defined by:
\begin{equation}
    \label{put}
\rho_k:=\displaystyle\sup_{\lambda,\sigma_j}\,\{\,\lambda:\: f-\lambda=\sigma_0+\sum_{j=1}^m\sigma_jg_j,\ \sigma_j\in \Sigma[x]\,,\,\deg(\sigma_jg_j)\le 2 k\}.\end{equation}
For each fixed $k$, \eqref{put} is a semidefinite program and therefore can be solved efficiently. 
Moreover, by invoking Putinar's Positivstellensatz, one obtains the convergence $\rho_k\uparrow f^\star$
as $k$ increasees.
In Table \ref{tab:top.Positivstellensatz} are listed several useful Positivstellens\"atze
that guarantee convergence of similar sequences $(\rho_k)_{k\in\N}$ to $f^\star$ (where now in \eqref{put} 
one uses the appropriate positivity certificate).
\begin{table}
    \caption{\small Several Positivstellens\"atze applicable in practice.}
    \label{tab:top.Positivstellensatz}
\footnotesize
\begin{center}
\begin{tabular}{|m{1.3cm}|m{7.5cm}|m{1.7cm}|}
\hline 
Author(s) & Statement & Application(s)\\ 
\hline 
Schm\"udgen \cite{schmudgen1991thek}&
If $f$ is positive on $S(g)$ and $S(g)$ is compact, then $f=\sum_{\alpha\in \{0,1\}^m}\sigma_\alpha \prod_{j=1}^m g_j^{\alpha_j}$ for some $\sigma_\alpha\in\Sigma[x]$.
 & \cite{helton2012semidefinite}\\
 \hline
Putinar \cite{putinar1993positive}&
If a polynomial $f$ is positive on $S(g)$ satisfying Archimedian assumption\footnotemark, then $f=\sigma_0+\sum_{j=1}^m\sigma_jg_j$ for some $\sigma_j\in\Sigma[x]$.
 & \cite{lasserre2001global}\\
 \hline
 Reznick \cite{reznick1995uniform}&
 If $f$ is a positive definite form, then $\|x\|^{2k}_2f\in\Sigma[x]$ for some $k\in\N$.&
 \cite{ahmadi2017construction}\\
 \hline
 Polya \cite{polya1974positive}&
 If $f$ is a homogeneous form and $f>0$ on $\R_+^n\backslash \{0\}$, then $(\sum_j x_j)^kf$ has nonnegative coefficients for some $k\in\N$.&
 \cite{de2002approximation}\\
  \hline
   Krivine-Stengle \cite{krivine1964anneaux, stengle1974nullstellensatz}&
If a polynomial $f$ is positive on $S(g)$, $S(g)$ is compact and $g_j\le 1$ on $S(g)$, then $f=\sum_{\alpha,\beta\in \N^m}c_{\alpha,\beta} \prod_{j=1}^m (g_j^{\alpha_j}(1-g_j)^{\beta_j})$ for some $c_{\alpha,\beta}\ge 0$.&
 \cite{lasserre2017bounded}\\
 \hline
 Putinar-Vasilescu \cite{putinar1999solving}&
 If a polynomial $f$ is nonnegative on $S(g)$, then for every $\varepsilon>0$, there exists $k\in\N$ such that $\theta^k(f+\varepsilon\theta^d)=\sigma_0+\sum_{j=1}^m\sigma_j g_j$ for some  $\sigma_j\in\Sigma[x]$, where $d:=1+\lfloor \deg(f)/2\rfloor$ and $\theta:=\|x\|^2_2+1$.&
 \cite{mai2019sums}\\
 \hline
\end{tabular}    
\end{center}
\end{table}
\footnotetext{There are $ \sigma_j\in\Sigma[x]$ such that $S(\{\sigma_0+\sum_{j=1}^m\sigma_jg_j\})$ is compact.}
However their associated  so-called {\em dense hierarchies} of linear/SDP programs  are only suitable for modest size POPs (e.g., $n\le 10$ and $\deg(f),\deg(g_j)\le 10$). 
Indeed, for instance, even though \eqref{put} is a semidefinite program, it involves $\binom{n+2k}{n}$ variables and semidefinite matrices of size up to 
$\binom{n+k}{n}$, a clear limitation  for state-of-the-art semidefinite solvers. 

\begin{center}
    \emph{Therefore a scientific challenge with important computational implications is to develop alternative positivity certificates that scale well in terms of computational complexity, at least in some identified class of problems.}
\end{center}

Fortunately as we next see, we can provide such alternative positivity certificates for the class of problems 
where some structured sparsity pattern is present in the problem description
(as often the case in large-scale problems). Indeed this sparsity pattern can be exploited to yield a positivity certificate in which the sparsity pattern is reflected, thus with potential significant computational savings.

\paragraph{Exploiting sparsity pattern.}
For $n,m\in \N^{>0}$, let $I:=\{1,\dots,n\}$ and $J:=\{1,\dots,m\}$. For $T\subset I$, denote by $\R[x(T)]$ (resp. $\Sigma[x(T)]$) the ring of polynomials (resp. the subset of SOS polynomials) in the variables $x(T):=\{x_i:i\in T\}$. 
Also denote by $\R[x(T)]_t$ (resp. $\Sigma[x(T)]_t$) the restriction of $\R[x(T)]$ (resp. $\Sigma[x(T)]$) to polynomials of degree at most $t$ (resp. $2t$). For $R\subset J$, we note $g_R:=\{g_j\,:\,j\in R\}$.

\emph{Designing alternative hierarchies for solving $f^\star:=\inf\{f(x):\,x\in S(g)\}$, significantly (computationally) cheaper than their dense
version \eqref{put}, while maintaining convergence to the optimal value $f^\star$ is a real challenge with important implications.}

One first such successful contribution is due to Waki et al. \cite{waki2006sums} when the input polynomial data $f,g_j$ are sparse, where by sparse we mean the following:
\begin{assumption}\label{assum:sparse.conditions}
The following conditions hold:
\begin{enumerate}
\item[(i)] Running intersection property (RIP): ${I} = \bigcup_{l = 1}^p {{I_l}} $ with $p\in \N^{\ge 2}$, $I_l\ne \emptyset$, $l=1,\dots,p$, and for every $l\in\{ 2,\dots,p\}$, there exists $s_l \in \{ 1,\dots,l-1 \}$, such that $\hat I_l\subset I_{s_l}$, where $\hat I_l:={I_{l}} \cap \left( {\bigcup_{j = 1}^{l-1} {{I_j}} } \right)$.
W.l.o.g, set $s_2:=1$ and  $\hat I_1:=\emptyset$. Denote $n_l:=|I_l|$ and $\hat n_l:=|\hat I_l|$, $l=1,\dots,p$.
\item[(ii)]  Structured sparsity pattern for the objective function\footnote{If there are $f_l$ in the sum $f$ such that $\deg(f_l)>\deg(f)$, we can always remove the high degree redundant term in $f_l$ which cancel with each other to make degree of $f_l$ at most $\deg(f)$}: $f = \sum_{l = 1}^p {{f_l}} $ where $f_l\in\R[x(I_l)]_{\deg(f)}$, $l=1,\dots,p$.
\item[(iii)]  Structured sparsity pattern for the constraints: $J = \bigcup_{l = 1}^p {{J_l}} $ and for every $j\in J_l$, $g_j\in\R[x(I_l)]$, $l=1,\dots,p$.
\item[(iv)] Additional redundant quadratic constraints: There exists $L>0$ such that $\|x\|_2^2 \le  L$ for all $x\in S(g)$ and ${L} - \|x(I_l)\|_2^2  \in g_{J_l}$, $l=1,\dots,p$.
\end{enumerate} 
\end{assumption}
With $\tau$ ($\leq n$) being the maximum number of variables appearing in each index subset $I_l$ of $f,g_j$, i.e., $\tau:=\max \{n_l\,:\, l=1,\dots,p\}$, Table \ref{tab:compare.sparse.dense} displays the respective computational complexity of the sparse hierarchy of Waki et al. \cite{waki2006sums} and the dense hierarchy of Lasserre \cite{lasserre2001global} for SDPs  with same order $k\in\N$.
\begin{table}
    \caption{\small Comparing the computational complexity of the sparse and dense hierarchies.}
    \label{tab:compare.sparse.dense}
\footnotesize
\begin{center}
\begin{tabular}{|m{3.5cm}|m{2.1cm}|m{2.1cm}|}
\hline 
SDP of order $k$ &  sparse hierarchy & dense hierarchy\\ 
\hline 
number of variables &
$O(\tau^{2k})$ & $O(n^{2k})$\\
 \hline 
largest size of SDP matrix &
$O(\tau^{k})$ & $O(n^k)$\\
 \hline
\end{tabular}    
\end{center}
\end{table}
Obviously the sparse hierarchy provides a potentially high  computational saving when compared to the dense one.
In addition, convergence of the hierarchy of Waki et al. to the optimal value of the original POP
was proved in \cite{lasserre2006convergent}, resulting in the following sparse version of Putinar's Positivstellensatz:
\begin{theorem}(Lasserre, Waki et al.) 
\label{theo:sparse.Putiner.rep}
Let Assumption \ref{assum:sparse.conditions} holds.
If a polynomial $f$ is positive on $S(g)$, then there exist ${\sigma _{0,l}} \in \Sigma {{[ x(I_l) ]}_{ k}}$, $ {\sigma _{j,l}} \in \Sigma {{[ x(I_l) ]}_{k - u_j}}$ with $u_j:=\lceil \deg(g_j)/2\rceil$, $ j\in J_l$, $  l=1,\dots,p$ such that 
\begin{equation}\label{eq:sparse.Putinar.Pos}
f =\sum\limits_{l = 1}^p {\left({\sigma _{0,l}} + \sum\limits_{j \in {J_l}} {{\sigma _{j,l}}{g_j}}\right)}\,.
\end{equation}
\end{theorem}
\emph{Compactness} of the feasible set $S(g)$ is a crucial ingredient of the proof in \cite{lasserre2006convergent}; shortly after, Grimm et al. \cite{grimm2007note} provided another (simpler) proof where ${\rm int}(S(g))\neq\emptyset$ is not needed, but where compactness of $S(g)$ is still a crucial assumption.

\paragraph{Motivation for sparse representations on non-compact sets.}
 We remark that Theorem \ref{theo:sparse.Putiner.rep} requires the additional redundant quadratic constraints (Assumption \ref{assum:sparse.conditions} $(iv)$), which is slightly stronger than just assuming the compactness of $S(g)$. 
When $S(g)$ is compact, we can always add these constraints but we need to know the radius $L>0$ of a ball centered at the origin and containing $S(g)$.
In this case, adding such constraints increases the number of positive semidefinite matrices from $m$ to $m+p$ in each 
SDP. In addition, it may be hard to verify compactness of $S(g)$ and obtain such a radius $L$.

To the best of our knowledge, in the non-compact case there is still 
no Positivstellensatz allowing one to build hierarchies for POPs satisfying :

- the RIP and the structured sparsity pattern from Assumption \ref{assum:sparse.conditions} $(i)$-$(iii)$,

- and a guarantee of convergence to the global optimum.

In fact we provide examples \ref{exam:zero.parameter.epsilon}, \ref{exam:zero.parameter.epsilon.noncommutative}, and \ref{exam:sum.pos.on.sem.set}, which show that in both unconstrained and constrained cases, there exist sparse nonnegative polynomials which do \emph{not} have a sparse SOS-based decomposition
\eqref{eq:sparse.Putinar.Pos} \`a la Putinar.
Such examples have been our motivation to investigate existence of sparse representations in the non-compact case, as well as to construct converging SDP-hierarchies for sparse polynomial optimization in general.
\paragraph{Dense rational SOS representations and non-compact POPs.}
In his famous and seminal work \cite{hilbert1888darstellung}, Hilbert characterized all cases where nonnegative polynomials are SOS of polynomials.
In 1927, Artin proved in~\cite{artin1927zerlegung} that every nonnegative polynomial can be decomposed as an SOS of {\em rational functions} (or {\em rational SOS} for short), thereby solving Hilbert's 17th problem.  
Namely, a polynomial $f$ is nonnegative if and only if there exist $\sigma_1, \sigma_2 \in \Sigma[x]$ such that 
$ f = \sigma_1/\sigma_2$.

Of course one can use Hilbert-Artin's representation to obtain a hierarchy of lower bounds for unconstrained POPs: $f^\star := \inf_{x \in \R^n} f(x)$, by computing $\rho_k:=\sup\{\lambda\in \R \,:\,\sigma_2(f-\lambda)=\sigma_1,\ \sigma_j\in\Sigma[x]_k\}$, for every $k\in \N$, so that $\rho_k\leq \rho_{k+1}\leq f^\star$ for all $k$.
However for each $k$ the resulting optimization problem is not an SDP (and not even convex) because of the nonlinear 
term $\sigma_2 \lambda$. (Even with an iterative dichotomy procedure on $\lambda$, one is left with an SDP hierarchy for \emph{each} fixed $\lambda$.)

When $f$ is a positive definite form Reznick proposes to select a so-called {\em uniform denominator}  in the Hilbert-Artin's representation, namely to replace $\sigma_2$ by some power of $\|x\|_2^2$ (see Table \ref{tab:top.Positivstellensatz}). 
As a result one obtains a decomposition in SOS of rational functions  for any arbitrary small perturbation of a nonnegative polynomial $f$ as follows: 
For every $\varepsilon>0$, there exists $k\in\N$ such that $\theta^k(f+\varepsilon\theta^d)=\sigma_0$ for some  $\sigma_0 \in \Sigma[x]$, with $d:=\lceil \deg(f)/2\rceil$ and $\theta:=\|x\|^2_2+1$. 
For abitrary $\varepsilon > 0$ fixed, we obtain an SDP-based hierarchy of bounds $\rho_k(\varepsilon)\:=\sup\{\lambda\in \R \,:\,\theta^k(f-\lambda+\varepsilon\theta^d)=\sigma_0\,,\, \sigma_0\in\Sigma[x]_{k+d}\}$, for every $k\in \N$. 
If $f^\star$ is attained then the sequence $(\rho_k(\varepsilon))_{k\in\N}$ converges to a value in a neighborhood of $f^\star$.
A similar idea, now based on Putinar-Vasilescu's Positivstellensatz \cite{putinar1999solving}, can be applied for polynomials nonnegative on non-compact basic semialgebraic sets (see Table \ref{tab:top.Positivstellensatz}).

This shows that rational SOS representations with \emph{fixed} forms for denominators are highly useful and applicable in non-compact POPs.
\paragraph{Contribution.} 
Our contribution is twofold:
\begin{itemize}
\item We first provide a rational SOS representation for a positive definite rational form which is a sum of sparse rational functions with uniform denominators, satisfying the structured sparsity pattern and the RIP stated in Assumption \ref{assum:sparse.conditions} $(i)$. 
This representation is provided in Theorem \ref{theo:rep.sum.pd}. 
As a direct consequence, we obtain a sparse version of Reznick's Positivstellensatz in Corollary \ref{coro:sparse.reznick}.
\item Then, we provide two positivity certificates for arbitrary small perturbations of -- globally nonnegative polynomials in Corollary \ref{coro:rep.sum.positive} --  and polynomials nonnegative on a (possibly non-compact) basic semialgebraic set in Corollary \ref{coro:rep.sum.positive.con},  when the input data satisfy a similar sparsity pattern. 
These two certificates are obtained via a sparse version of Putinar-Vasilescu's Positivstellensatz and do not require the additional constraints from Assumption \ref{assum:sparse.conditions} $(iv)$.
\end{itemize}
Illustrations of such positivity certificates for polynomials nonnegative on non-compact basic semialgebraic sets are 
provided in Example \ref{ex:homogeneous.exam}, \ref{exam:zero.parameter.epsilon}, \ref{exam:zero.parameter.epsilon.noncommutative} and \ref{exam:sum.pos.on.sem.set}, for which
positvity certificates \eqref{put} do not exist.
The existence of such sparse SOS-representations is proved by combining different tools:
\begin{itemize}
\item First, we use an idea similar to that developed in Grimm et al. \cite{grimm2007note} (in the compact case) to prove that a sparse positive definite form can be decomposed as SOS of sparse positive definite rational forms; 
as expected the non-compact case is rechnically more involved.
This yields a \emph{sparse version} of Hilbert-Artin's representation theorem in the case of positive definite forms.
\item Next, we use generalizations of Schm\"udgen's Positivstellensatz presented by Schweighofer \cite{schweighofer2003iterated}, Berr-W{\"o}rmann \cite{berr2001positive}, Jacobi \cite{jacobi2001representation}, and Marshall \cite{marshall2001extending, marshall2002general}, for a finitely generated $\R$-algebra  in each term of the sum, to obtain again a \emph{sparse version}, this time of Reznick's Positivstellensatz for positive definite forms.
\item Finally we combine the homogenization/dehomogenization method that we already used in \cite{mai2019sums} together  with limit tools, to provide the two sparse versions of Putinar-Vasilescu's Positivstellensatz.
\end{itemize}
\section{Main results}
 For $(i,j)\in\N^2$, we denote the Kronecker delta function by
\[\delta_{i,j} := \left\{ \begin{array}{rl}
1 &\text{ if }i=j\,,\\
0&\text{ if }i\ne j\,.
\end{array} \right.\]
When Assumption \ref{assum:sparse.conditions} $(i)$ holds, define:
\[\Phi_l:=\begin{cases}{\|x(\hat I_l)\|_2^{2(1-\delta_{l,1})}\prod_{j = l+1}^{p}{ \|x(\hat I_j)\|_2^{2\delta_{l,s_j}}}} &\text{ if } l=1,\dots,p-1\,,\\
\|x(\hat I_l)\|_2^{2(1-\delta_{l,1})}&\text{ if }l=p\,.\end{cases}\] 
Obviously, one has $\Phi_l\in \R[x( I_l)]$, for each $l=1,\dots,p$.
Let us state the first main result of this paper which yields a sparse version of Reznick's PositivStellensatz as a particular case.
\begin{theorem}\label{theo:rep.sum.pd}
Let Assumption \ref{assum:sparse.conditions} $(i)$ holds.
Let $f\in\R(x)$ be a positive definite rational form of degree $2d$ with $d\in \N^{>0}$ such that 
\[f=\sum_{l=1}^p \frac{p_l}{\|x( I_l)\|_2^{2k_l}}\,,\]
where $p_l\in\R[x(I_l)]$ is homogeneous of degree $2(d+k_l)$ for some $k_l\in \N$, $l=1,\dots,p$. 
Then there exist $k\in\N$ and $\sigma_l\in \Sigma[x( I_l)]_{d+k (1+\deg(\Phi_l)/2)}$, $l=1,\dots,p$,  such that 
\begin{equation}\label{eq:homogeneous}
f=\sum\limits_{l = 1}^p {\frac{\sigma_l}{\|x( I_l)\|_2^{2k}\Phi_l^k}}\,.
\end{equation}
\end{theorem}
The proof of Theorem \ref{theo:rep.sum.pd} can be found in Section \ref{proof:rep.sum.pd}.
As a consequence, we obtain the following sparse version of Reznick's Positivstellensatz:
\begin{corollary}\label{coro:sparse.reznick}
Let Assumption \ref{assum:sparse.conditions} $(i)$ holds.
Assume that $f$ is a positive definite form of degree $2d$ with $d\in \N^{>0}$ and $f = \sum_{l=1}^p f_l$, where $f_l\in\R[x(I_l)]$ is homogeneous of degree $2d$, $l=1,\dots,p$. 
Then there exist $k\in\N$ and $\sigma_l\in \Sigma[x( I_l)]_{d+k (1+\deg(\Phi_l)/2)}$, $l=1,\dots,p$,  such that 
\begin{equation}\label{eq:sparse.reznick}
f=\sum\limits_{l = 1}^p {\frac{\sigma_l}{H_l^k}}\,,
\end{equation}
where $H_l:=\|x( I_l)\|_2^{2}\,\Phi_l$, $l=1,\dots,p$.
\end{corollary}
To prove Corollary \ref{coro:sparse.reznick}, we apply Theorem \ref{theo:rep.sum.pd} with $k_l=0$, $l=1,\dots,p$.
The representation \eqref{eq:sparse.reznick} can still  hold even when $f$ is not a positive definite form, as illustrated in the following example: 
\begin{example}\label{ex:homogeneous.exam}
Let $f=f_1+f_2$, where
\[f_1:=x_4^2(x_1^4x_2^2+x_2^4x_3^2+x_1^2x_3^4-3x_1^2x_2^2x_3^2)+x_3^8\]
is the so-called Delzell's polynomial and $f_2:=x_1^2x_2^2x_3^2x_5^2$. 
The polynomial $f_1$ is nonnegative, but not SOS  as shown in \cite[Example 2]{leep2001polynomials}. 
Let $I_1:=\{1,2,3,4\}$ and $I_2:=\{1,2,3,5\}$.
Then $f_1\in \R[x(I_1)]$ and $f_2 \in \R[x(I_2)]$ are nonnegative and homogeneous of degree $8$. Since $f_1$ is nonnegative then $f$ is nonnegative. The following statements hold:
\begin{enumerate}
\item $f$ is a nonnegative form, but is not positive definite;
\item $f\notin \Sigma[x(I_1)]+\Sigma[x(I_2)]$, but $f\in \frac{\Sigma[x(I_1)]_6}{\|x( I_1)\|_2^{2}\Phi_1}+ \frac{\Sigma[x(I_2)]_6}{\|x( I_2)\|_2^{2}\Phi_2}$.
\end{enumerate}
The first statement follows from the fact that $f(0,0,0,1,1)=0$, ensuring that $f$ is not a positive definite form.

\noindent
Proof of the second statement:  
Assume by contradiction that $f=\sigma_1+\sigma_2$ for some $\sigma_l\in\Sigma[x(I_l)]$, $l=1,2$.
Evaluation at $x_5=0$ yields $f_1=\sigma_1+\sigma_2(x_1,x_2,x_3,0)$, so that $f_1$ is an SOS, which is impossible. 
Thus, $f\notin \Sigma[x(I_1)]+\Sigma[x(I_2)]$. 
However,  $(x_1^2+x_2^2+x_3^3)f_1$ is SOS  according to \cite[Example 4.4]{schabert2019uniform}, so $(x_1^2+x_2^2+x_3^2)f\in \Sigma[x(I_1)]_{5}+\Sigma[x(I_2)]_5$. 
Note that $\Phi_1=\Phi_2=x_1^2+x_2^2+x_3^2$. 
Therefore
\[f\in \frac{\Sigma[x(I_1)]_{5}}{\Phi_1}+\frac{\Sigma[x(I_2)]_5}{\Phi_2} \subset \frac{\Sigma[x(I_1)]_6}{H_1}+ \frac{\Sigma[x(I_2)]_6}{H_2}\,.\]
\end{example}

 When Assumption \ref{assum:sparse.conditions} $(i)$ holds, define the following polynomials, for each $l=1,\dots,p$:
\begin{itemize}
\item $\theta_l:=\|x(I_l)\|^2_2+1$ and $ \hat \theta_l:=\|x(\hat I_l)\|^2_2+1$;
\item $D_l:=\begin{cases} {\hat \theta_l^{1-\delta_{l,1}}\prod_{j = l+1}^{p}{ \hat \theta_j^{\delta_{l,s_j}}}}&\text{ if }l < p\,,\\
\hat \theta_l^{1-\delta_{l,1}}&\text{ if }l=p\,;
\end{cases}$
\item $\Theta_l:=\theta_l D_l$ and $\omega_l:=\deg(\Theta_l)/2$.
\end{itemize}
Note that $\Theta_l\in\Sigma[x(I_l)]_{\omega_l}$, for each $l=1,\dots,p$.

We next state the following \emph{sparse version} of Putinar-Vasilescu's Positivstellensatz for 
polynomials nonnegative on the whole $\R^n$.
\begin{corollary}\label{coro:rep.sum.positive}
Let $f$ be a nonnegative polynomial such that the conditions $(i)$ and $(ii)$ of Assumption \ref{assum:sparse.conditions} hold.
Let $\varepsilon>0$ and $d\ge  \deg(f)/2 $. 
Then there exist $k\in\N$ and   $\sigma_l\in \Sigma[x( I_l)]_{d+k\omega_l}$, $l=1,\dots,p$,  such that 
\begin{equation}\label{eq:dehomogeneous}
f+\varepsilon \sum\limits_{l = 1}^p {\theta_l^{d}}=\sum\limits_{l = 1}^p {\frac{\sigma_l }{\Theta_l^k}}\,.
\end{equation}
\end{corollary}
The proof of Corollary \ref{coro:rep.sum.positive} is postponed to Section \ref{proof:rep.sum.positive}.

The representation \eqref{eq:dehomogeneous} can still  hold even if $\varepsilon=0$, as illustrated in the following examples: 
\begin{example}\label{exam:zero.parameter.epsilon}
Let $f=f_1+f_2$, where 
\[f_1:=8+\frac{1}{2} x_1^2 x_2^4+(x_1^2 - 2 x_1^3)x_2^3 + (2x_1 +10 x_1^2 +4x_1^3+3x_1^4)x_2^2+4(x_1-2x_1^2)x_2\]
is the so-called Leep-Starr's polynomial and $f_2:=x_1^2x_3^2$. 
Let $I_1:=\{1,2\}$ and $I_2:=\{1,3\}$, so that 
 $f_1\in \R[x(I_1)]$ and $f_2\in \R[x(I_2)]$. 
 As shown in \cite[Example 2]{leep2001polynomials}, $f_1$ is nonnegative but not an SOS. 
In addition, $f$ is nonnegative.

We claim that $f\notin \Sigma[x(I_1)]+\Sigma[x(I_2)]$. 
Indeed, assume by contradiction that $f=\sigma_1+\sigma_2$ for some $\sigma_l\in\Sigma[x(I_l)]$, $l=1,2$. Evaluation at $x_3=0$, yields $f_1=\sigma_1+\sigma_2(x_2,0)$, so that $f_1$ is an SOS, which is impossible. 

However,  $(x_1^2+1)^2f_1$ is a sum of three squares of polynomials according to \cite[Example 2]{leep2001polynomials}, so $(x_1^2+1)^2f\in \Sigma[x(I_1)]_{5}+\Sigma[x(I_2)]_5$. 
Note that $D_1=D_2=x_1^2+1$. 
Thus, 
\[f\in \frac{\Sigma[x(I_1)]_{5}}{D_1^2}+\frac{\Sigma[x(I_2)]_5}{D_2^2} \subset \frac{\Sigma[x(I_1)]_{7}}{\Theta_1^2}+\frac{\Sigma[x(I_2)]_7}{\Theta_2^2}\,.\]
\end{example}
\begin{example}\label{exam:zero.parameter.epsilon.noncommutative}
As shown in  \cite[Example 5.2]{klep2019sparse}, the nonnegative polynomial
\[f=x_1^2-2x_1x_2+3x_2^2-2x_1^2x_2+2x_1^2x_2^2-2x_2x_3+6x_3^2+18x_2^2x_3-54x_2x_3^2+142x_2^2x_3^2\]
satisfies $f\in \R[x(I_1)]+\R[x(I_2)]$ and $f\notin \Sigma[x(I_1)]+\Sigma[x(I_2)]$, with $I_1=\{1,2\}$ and $I_2=\{2,3\}$. However, $f\in \frac{\Sigma[x(I_1)]_4}{\Theta_1}+\frac{\Sigma[x(I_2)]_4}{\Theta_2}$, where $\Theta_1=(x_2^2+1)(x_1^2+x_2^2+1)$ and $\Theta_2=(x_2^2+1)(x_2^2+x_3^2+1)$. It is due to the fact that $f= \frac{\sigma_1}{D_1}+\frac{\sigma_2}{D_2}$, where $D_1=D_2=x_2^2+1$ and $\sigma_1$ and $\sigma_2$ are SOS polynomials given in Appendix \ref{sec:Appendix}.
\end{example}
We next state our second main result, namely a \emph{sparse version} of Putinar-Vasilescu's Positivstellensatz for polynomials nonnegative on (possibly non-compact) basic semialgebraic sets.
\begin{corollary}\label{coro:rep.sum.positive.con}
Let $f\in\R[x]$ be nonnegative on $S(g)$ such that the conditions $(i)$, $(ii)$ and $(iii)$ of Assumption \ref{assum:sparse.conditions} hold.
Let $\varepsilon>0$ and $d\ge 1+\lfloor \deg(f)/2\rfloor$. 
Recall that $u_j=\lceil \deg(g_j)/2\rceil$, for all $j=1,\dots,m$.
Then there exist $k\in\N$,   $\sigma_{0,l}\in \Sigma[x( I_l)]_{d+k\omega_l}$ and $\sigma_{j,l}\in \Sigma[x( I_l)]_{d+k\omega_l-u_j}$, $j\in J_l$, $l=1,\dots,p$,  such that 
\begin{equation}\label{eq:dehomogeneous.con}
f+\varepsilon \sum\limits_{l = 1}^p {\theta_l^{d}}=\sum\limits_{l = 1}^p {\frac{{{\sigma_{0,l} }}+\sum_{j\in J_l}{{\sigma_{j,l} {g_j} }} }{\Theta_l^{k}}}\,.
\end{equation}
\end{corollary}
The proof of Corollary \ref{coro:rep.sum.positive.con} is postponed to Section \ref{proof:rep.sum.positive.con}.
\begin{example}\label{exam:sum.pos.on.sem.set}
Let $f=f_1+f_2$, where $f_1=x_1x_2$ and $f_2=x_2^2x_3$. Let $g=\{g_1,g_2,g_3\}$, where $g_1=x_2^3$, $g_2=-g_1$ and $g_3=x_3$. It is not hard to show that $f= 0$ on $S(g)$, so that $f\ge 0$ on $S(g)$. By noting $I_1:=\{1,2\}$ and $I_2:=\{2,3\}$, one has $\{f_1,g_1,g_2\}\subset \R[x(I_1)]$ and $\{f_2,g_3\}\subset \R[x(I_2)]$. We claim the following statements:
\begin{enumerate}
\item $f\notin \Sigma[x(I_1)]+g_1\R[x(I_1)]+\Sigma[x(I_2)]+g_3\Sigma[x(I_2)]$;
\item for every $\varepsilon>0$, 
\[f+\varepsilon (\theta_1^2+\theta_2^2)\in \frac{\Sigma[x(I_1)]_{2k+2}+g_1\R[x(I_1)]_{4k+1}}{\Theta_1^k}+\frac{\Sigma[x(I_2)]_{2k+2}+g_3\Sigma[x(I_2)]_{4k+3}}{\Theta_2^k}\,,\]
for some $k\in\N$ depending on $\varepsilon$.
\end{enumerate}
Proof of the first statement: 
Assume by contradiction that there exist $\sigma_1\in \Sigma[x(I_1)]$, $\psi_1\in \R[x(I_1)]$ and $\sigma_2, \sigma_3\in \Sigma[x(I_2)]$ such that $f=\sigma_1+\psi_1 g_1+\sigma_2+\sigma_3 g_3$. 
Evaluation at $x_1=1$ and $x_3=0$ yields
\[x_2=\sigma_1(1,x_2)+\psi_1(1,x_2) x_2^3+\sigma_2(x_2,0)\in \Sigma[x_2]+x_2^3\R[x_2]\,,\]
which is impossible due to \cite[Lemma 3.3 (i)]{mai2019sums}. 

\noindent
Proof of the second statement:
With $\varepsilon>0$ fixed,
\[f_1+\varepsilon \theta_1^2 =x_1x_2+\varepsilon(1+x_1^2+x_2^2)^2=x_1x_2+\varepsilon+\varepsilon x_1^2 +\sigma_4\,,\]
for some $\sigma_4\in\Sigma[x(I_1)]_2$. 
Let $k\in\N^{\ge 2}$ be fixed.
Then $D_1^k=(1+x_2^2)^k=1+kx_2^2+x_2^4\sigma_5$ for some $\sigma_5\in\Sigma[x_2]_{k-2}$, which
implies 
\[D_1^k(f_1+\varepsilon \theta_1^2)=x_1x_2+\varepsilon x_1^2 + \varepsilon k x_2^2+\sigma_6+\psi_2 x_2^3\,,\]
for some $\sigma_6\in\Sigma[x(I_1)]_{k+2}$ and $\psi_2\in \R[x(I_1)]_{2k+1}$. 
Assume that $k\ge \varepsilon^{-2}/4$. Then
\[\begin{array}{rl}
D_1^k(f_1+\varepsilon \theta_1^2)&=x_1^2\left(\varepsilon  - \frac{1}{{4\varepsilon k}}\right) + {\left( {x_2\sqrt {\varepsilon k }  + \frac{x_1}{{2\sqrt {\varepsilon k} }}} \right)^2}+\sigma_6+\psi_2 x_2^3\\
&\in \Sigma[x(I_1)]_{k+2}+g_1\R[x(I_1)]_{2k+1}\,,
\end{array}\]
which implies $f_1+\varepsilon \theta_1^2\in \frac{\Sigma[x(I_1)]_{2k+2}+g_1\R[x(I_1)]_{4k+1}}{\Theta_1^k}$. 
We also have 
\[f_2+\varepsilon \theta_2^2\in\frac{\Sigma[x(I_2)]_{2k+2}+g_3\Sigma[x(I_2)]_{4k+3}}{\Theta_2^k}\]
since $f_2\in g_3\Sigma[x(I_2)]_1$, proving the second statement.
\end{example}
\section{Preliminary material}
%
%
Given $\alpha = (\alpha_1,\dots,\alpha_n) \in \N^n$, we note $|\alpha| := \alpha_1 + \dots + \alpha_n$ and $x^\alpha:=x_1^{\alpha_1}\dots x_n^{\alpha_n}$.
Let $(x^\alpha)_{\alpha\in\N^n}$ be the canonical basis of monomials for  $\R[x]$ (ordered according to the graded lexicographic order) and 
$v_t(x)$ be the vector of monomials up to degree $t$, with length $s(t) = {\binom{n+t}{n}}$.
A polynomial $h\in\R[x]_t$ is written as  
$h(x)\,=\,\sum_{| \alpha | \leq t} h_\alpha\,x^\alpha\,=\,\mathbf{h}^Tv_d(x)$, 
where $\mathbf{h}=(h_\alpha)\in\R^{s(t)}$ is its vector of coefficients in the canonical basis.
Denote by $\S^{n-1} := \{x \in \R^n : \|x\|_2 = 1 \}$ the $(n-1)$-dimensional unit sphere.

A function $h$ is \emph{homogeneous} of degree $t$ if
$h(\lambda x)=\lambda^t h(x)$ for all $x\in\R^n$ and each $\lambda\in\R$. Therefore a homogeneous polynomial can be written as 
$h = \sum_{| \alpha | = t} {h_\alpha x^\alpha }$. A function $f:\R^n\to\R$ is even if $f(x)=f(-x)$ for all $x$.
A rational function $h$ is the  ratio of two polynomials and denote by $\R(x)$ the space of all rational functions. 
A homogeneous rational function (also called be a {\em rational form}, or form in
short) can be written as the ratio of two homogeneous polynomials.

The degree-$d$ {\em homogenization} $\tilde h$ of $h\in \R(x_1,\dots,x_n)$ is a  homogeneous rational function in $\R(x_1,\dots,x_{n+1})$  of degree $d$ defined by 
$\tilde h ( x,x_{n + 1}) = x_{n + 1}^{d} h ( x/x_{n + 1} )$.
A {\em rational positive definite form} of degree $t$ is a homogeneous rational function of degree $t$ which is positive everywhere except at the origin.
Equivalently, a homogeneous rational function $h$ of degree $t$ is a rational positive definite form of degree $t$ if and only if there exists $\varepsilon>0$ such that $h\ge \varepsilon \|x\|_2^{2t}$.

We briefly recall some algebraic tools from generalizations of Schm\"udgen's Positivstellensatz \cite{schweighofer2003iterated} which will be used in the sequel. 
An associative algebra $A$ is called {\em a finitely generated $\R$-algebra} if there exists a finite set of elements $a_1,\dots,a_n$ of $A$ such that every element of $A$ can be expressed as a polynomial in $a_1,\dots,a_n$, with coefficients in $\R$.
Let $A$ be a commutative ring.
We denote by $\Sigma A^2$ the set of all SOS of elements in $A$.
A subset $T$ of $A$ is called a {\em preordering} if $T$ contains all squares
and is closed under addition and multiplication. 
The preordering $T$ generated by elements $t_1,\dots,t_m$ (so-called smallest preordering containing $t_1,\dots,t_m$) consists of all elements of the form $\sum_{\alpha\in\{0,1\}^m}(\sigma_\alpha \prod_{j=1}^m t_j^{\alpha_j})$, with $\sigma_\alpha\in \Sigma A^2$. 
The real spectrum of a ring $A$ with fixed preordering $T$, denoted by $\Sper_T A$, is defined by
\[\Sper_T A:=\{\varphi\in \Hom(A,\R): \varphi (T)\subset \R_+\}\,,\]
where $\Hom(A,\R)$ is the set of all ring homomorphisms from $A$ to $\R$. 
Let $A$ be a ring with fixed preordering $T$.
We denote by $H(A)$ (resp. $H'(A)$) the ring of geometrically (resp. arithmetically) bounded elements in $A$, i.e.,
\[\begin{array}{rl}
&H(A):=\{ h\in A:\,\exists K\in \N:\,K\pm h\ge 0 \text{ on }\Sper_T A\}\\
&H'(A):=\{h\in A:\,\exists K\in \N:\,K\pm h \in T\}\,,
\end{array}\]
where ``$h\ge 0$ on $\Sper_T A$" means ``$\varphi (h)\ge 0$ for all $\varphi\in \Sper_T A$". 
From \cite[(1.1)]{schweighofer2003iterated}, 
\begin{equation}\label{eq:geometrically.implies.arithmetically}
A=H(A)\Rightarrow A=H'(A)\,.
\end{equation} 
Let us restate \cite[Theorem 1.3]{schweighofer2003iterated} as follows:
\begin{lemma}\label{lem:positive.in.preordering}
If $\Q \subset A$ and $A = H'(A)$, then for any $f\in A$,
\[f>0\text{ on }\Sper_T A \Rightarrow f\in T\,.\]
\end{lemma}
Let us note $\|h\|_1:=\sum\nolimits_\alpha  {{|h_\alpha |}} $ for a given $h \in \R[x]$. 
We start with two preliminary results.
\begin{lemma}\label{lem:continuity.homogeneous}
For $k \in \N$ and $d\in\N^{> 0}$, let $q$ be a form of degree $2(d+k)$ and $f=\frac{q}{\|x\|_2^{2k}}\in\R( x)$. 
Then  $f$ is continuous and homogeneous of degree $2d$.
\end{lemma}
\begin{proof}
The rational function $f$ is obviously homogeneous of degree $2d$. 
To show that $f$ is continuous, it is sufficient to prove that $f$ is continuous at zero. 
Let $y\in \S^{n-1}$, then one has $| y^{\alpha} | \leq 1$, for all $\alpha$ such that $| \alpha | = 2 (d+k)$. Thus,
\[\left|{q( y)}\right|=| \sum_{\alpha}{q_{\alpha}y^{\alpha}}| \le  \sum_{\alpha}{|q_{\alpha}| |y^{\alpha}|}\le  \sum_{\alpha}{|q_{\alpha}|}=\|q\|_1\,.\]
From this, one has $|f(y)|=|{{q(y)}}|\le {\|q\|_1}$.
Let $x \neq 0$. Since $f$ is homogeneous of degree $2d$, 
\[ \frac{|f(x)|}{ {\| x\|_2^{2d}} }     = \left| f\left(\frac{x}{\|x \|}   \right) \right| \le {\|q\|_1} \,.\]
Hence for all $x \neq 0$, $|f(x)| \leq {\|q\|_1} \| x\|_2^{2d}$, thus  $\lim_{ x\to 0}{f(x)}=0$,  yielding the conclusion.
\end{proof}
\begin{lemma}\label{lem:approx.form.on.sphere}
Let $h:\R^n\to \R$ be an even function such that $h$ is continuous on $\mathbb{S}^{n-1}$. 
Then there exists a sequence $(q_k)_{k \in \N}$ of homogeneous polynomials, with $deg(q_k)=2k$ for all $k \in \N$,  converging uniformly to $h$ on $\mathbb{S}^{n-1}$.
\end{lemma}
\begin{proof}
We rely on \cite[Theorem 1.4 (b)]{varju2007approximation} with $K=\mathbb{S}^{n-1}$ and the statement $(iii)$ in \cite[Proposition 1.2]{varju2007approximation}.
We first need to ensure that $\mathbb{S}^{n-1}$ is the boundary of a convex domain, which is obvious since it is the boundary of the unit ball. 
Then we use the fact that $\mathbb{S}^{n-1}$ is twice continuously differentiable and has Gaussian curvature $1$ at every point.
\end{proof}
%
\begin{lemma}\label{lem:p.equal.2}
Assume that $I=I_1\cup I_2$. Let $f\in\R(x)$ be a rational positive definite form of degree $2d$ with $d\in\N^{>0}$ such that $f=f_1+f_2$ with $f_1\in\R(x(I_1))$ and $f_2\in\R(x(I_2))$ being continuous and  homogeneous of degree $2d$. 
Then there exists a continuous rational function $\varphi\in\R(x(I_1\cap I_2))$ defined by
\[\varphi(y)= \frac{q(y)}{\|y\|_2^{2k}}\,,\, \forall y\in \R^{|I_1\cap I_2|}\,,\]
where $q\in \R[x(I_1\cap I_2)]$ is a form of degree $2(d+k)$ for some $k\in\N$ (only depending on $d$, $\varepsilon$ and $f_1$) such that 
\[f= h_1+h_2\,,\]
 where $h_1:=f_1-\varphi \in\R(x(I_1))$ and $h_2:=f_2+\varphi\in\R(x(I_2))$ are continuous rational positive definite forms of degree $2d$.
\end{lemma}
\begin{proof} 
Since $f\in\R(x)$ is a rational positive definite form of degree $2d$,  there exists $\varepsilon >0$ such that
\begin{equation}\label{eq:positive.form.pro1} 
f\ge \varepsilon\|x\|_2^{2d}\text{ on }\R^n \,.
\end{equation}
Let us define the function $h:\R^{|I_1\cap I_2|}\to \R$ by
\begin{equation}\label{eq:def.h}
h(y):=\min\{\psi(\xi,y):\, \xi\in \R^{|I_1\backslash I_2|}\}\,,
\end{equation} 
where $\psi(\xi,y):=f_1(\xi,y) -\frac{\varepsilon}2 \|(\xi,y)\|_2^{2d}$.
To show that $h$ is well-defined, it is sufficient to prove that $\xi\mapsto \psi(\xi,y)$ is coercive on $\R^{|I_1\backslash I_2|}$ with fixed $y\in\R^{|I_1\cap I_2|}$. 
Indeed, for all $\xi\in \R^{|I_1\backslash I_2|}$, by \eqref{eq:positive.form.pro1},
\[\frac \varepsilon 2 \|\xi\|_2^{2d}\le \frac \varepsilon 2 \|(\xi,y)\|_2^{2d} \le f(\xi,y,0)-\frac \varepsilon 2 \|(\xi,y)\|_2^{2d}=f_1(\xi,y)-\frac \varepsilon 2 \|(\xi,y)\|_2^{2d}+f_2(y,0)\,,\]
so $\psi(\xi,y)\ge \frac \varepsilon 2 \|\xi\|_2^{2d} - f_2(y,0)$.
Moreover, $h$ is homogeneous of degree $2d$. 
Indeed, for every $t\in\R \backslash \{0\}$, one has
\[\begin{array}{rl}
h(ty)=&\min\{f_1(\xi,ty)-\frac{\varepsilon}2 \|(\xi,ty)\|_2^{2d}:\, \xi\in \R^{|I_1\backslash I_2|}\}\\
=&t^{2d}\min\{f_1(\xi/t,y)-\frac{\varepsilon}2 \|(\xi/t,y)\|_2^{2d}:\, \xi\in \R^{|I_1\backslash I_2|}\}\\
=&t^{2d}\min\{f_1(\xi,y)-\frac{\varepsilon}2 \|(\xi,y)\|_2^{2d}:\, \xi\in \R^{|I_1\backslash I_2|}\}\\
=&t^{2d} h(y)\,.
\end{array}\]
To show that $h$ is continuous, let $y_1,y_2\in\R^{I_1\cap I_2}$. We choose $\xi_1,\xi_2\in \R^{|I_1\backslash I_2|}$ minimizing $\xi\mapsto \psi(\xi,y_1)$ and $\xi\mapsto \psi(\xi,y_2)$, respectively. Then\[\psi(\xi_1,y_1)-\psi(\xi_1,y_2)\le\psi(\xi_1,y_1)-\psi(\xi_2,y_2)\le \psi(\xi_2,y_1)-\psi(\xi_2,y_2)\,.\]
From this and by \eqref{eq:def.h},
\[\begin{array}{rl}
|h(y_1)-h(y_2)|=&|\psi(\xi_1,y_1)-\psi(\xi_2,y_2)|\\
\le &\max\{|\psi(\xi_1,y_1)-\psi(\xi_1,y_2)|, |\psi(\xi_2,y_1)-\psi(\xi_2,y_2)|\}
\,.
\end{array}\]
This shows that $h$ is uniformly continuous on every compact subset of $\R^{|I_1\cap I_2|}$ because $\psi$ is uniformly continuous on 
every compact subset of $\R^{|I_1|}$.  Next, we claim that 
\begin{equation}\label{eq:ineq.f1.f2}
f_1-h \ge \frac{\varepsilon}2 \|x(I_1)\|_2^{2d}  \text{ on } \R^{|I_1|}\ \text{ and }\ f_2+h \ge \frac{\varepsilon}2 \|x(I_2)\|_2^{2d}  \text{ on } \R^{|I_2|}\,.
\end{equation} 
The first claim is clear by the definition of $h$.
 To prove the second one, let $(y,z)\in\R^{|I_2|}=\R^{|I_1\cap I_2|}\times \R^{|I_2\backslash I_1|}$, and choose $\xi\in\R^{|I_1\backslash I_2|}$ such that $h(y)=f_1(\xi,y)-\frac{\varepsilon}2 \|(\xi,y)\|_2^{2d}$. 
By \eqref{eq:positive.form.pro1}, observe that
\[\begin{array}{rl}
f_2(y,z)+h(y)=&f_2(y,z)+f_1(\xi,y)-\frac{\varepsilon}2 \|(\xi,y)\|_2^{2d}=f(\xi,y,z)-\frac{\varepsilon}2 \|(\xi,y)\|_2^{2d} \\
\ge& {\varepsilon} \|(\xi,y,z)\|_2^{2d}-\frac{\varepsilon}2 \|(\xi,y)\|_2^{2d}\ge \frac{\varepsilon}2 \|(y,z)\|_2^{2d}\,.
\end{array}\]
Next, we will approximate $h$ by a form of even degree on $\S^{|I_1\cap I_2|-1}$. Note that $h$ is continuous and even since $h$ is homogeneous of even degree. From this and by using Lemma \ref{lem:approx.form.on.sphere}, there exists $q\in\R[x(I_1\cap I_2)]$ homogeneous  of degree $2K$ for some $K\ge d$ such that 
\[|q-h|\le \frac \varepsilon 4\text{ on }\S^{|I_1\cap I_2|-1}\,.\] 
Since $1=\|x(I_1\cap I_2)\|_2^2$ on $\S^{|I_1\cap I_2|-1}$, 
\[\left|\frac{q}{\|x(I_1\cap I_2)\|_2^{2(K-d)}}-h\right|\le \frac \varepsilon 4\text{ on }\S^{|I_1\cap I_2|-1}\,.\] 
From this and since $\frac{q}{\|x(I_1\cap I_2)\|_2^{2(K-d)}}-h$ is homogeneous of degree $2d$, one has
\begin{equation}\label{ineq:important}
\left|\frac{q}{\|x(I_1\cap I_2)\|_2^{2(K-d)}}-h\right|\le \frac \varepsilon 4 \|x(I_1\cap I_2)\|_2^{2d}\text{ on }\R^{|I_1\cap I_2|}\backslash \{0\}\,.
\end{equation} 
By setting $\varphi:=\frac{q}{\|x(I_1\cap I_2)\|_2^{2(K-d)}}$ and using Lemma \ref{lem:continuity.homogeneous}, $\varphi$ is continuous on $\R^{|I_1\cap I_2|}$.
By setting $h_1:=f_1 - \varphi \in \R(x(I_1))$ and $ h_2:=f_2 + \varphi \in \R(x(I_2))$, one has 
$f=h_1+h_2$. 
Let us prove that $h_1$ and $h_2$ are both  rational positive definite forms of degree $2d$.
Indeed, by \eqref{eq:ineq.f1.f2} and \eqref{ineq:important},
\[\begin{array}{rl}
h_1=(f_1 - h)+ \left( h-\varphi\right)\ge  \frac{\varepsilon}2 \|x(I_1)\|_2^{2d} - \frac \varepsilon 4 \|x(I_1\cap I_2)\|_2^{2d}\ge \frac{\varepsilon}4 \|x(I_1)\|_2^{2d}\,,\\
h_2=(f_2 + h)+ \left(\varphi-h \right)\ge  \frac{\varepsilon}2 \|x(I_2)\|_2^{2d} - \frac \varepsilon 4 \|x(I_1\cap I_2)\|_2^{2d}\ge \frac{\varepsilon}4 \|x(I_2)\|_2^{2d}\,.
\end{array}\]
Thus, $ h_l \ge \frac{\varepsilon}4 \|x(I_l)\|_2^{2d}$ on $ \R^{|I_l|}$, $l=1,2$.
By setting $k:=K-d$, the conclusion follows.
 \end{proof}
 
Building up on Lemma \ref{lem:p.equal.2}, the following helpful result provides a non-compact analogue of Grimm et al. \cite{grimm2007note} and as expected, the non-compact case is much more involved.
\begin{lemma}\label{lem:sparse.posi.form.general.case}
Let  Assumption \ref{assum:sparse.conditions} $(i)$ holds.
Let $f\in\R(x)$ be a rational positive definite form of degree $2d$ such that $f=\sum_{l=1}^p f_l$ with $f_l\in\R(x(I_l))$ being continuous and  homogeneous of degree $2d$, $l=1,\dots,p$.
Then there exist continuous rational functions $\varphi_l\in\R(x(\hat I_l))$, $l=2,\dots,p$, defined by
\[\varphi_l(y)= \frac{q_l(y)}{\|y\|_2^{2k_l}}\,,\, \forall y\in \R^{\hat n_l}\,,\,l=2,\dots,p\,,\]
where $q_l\in \R[x(\hat I_l)]$ is homogeneous of degree $2(d+k_l)$ for some $k_l\in\N$, $l=2,\dots,p$, such that 
\[ f=\sum\limits_{l = 1}^p h_l\,,\]
 where $h_l:=f_l+\varphi_l-\sum_{j = l+1}^{p} \delta_{l,s_j}\varphi_j \in\R(x(I_l))$, with $\varphi_1:=0$, is a continuous rational positive definite forms of degree $2d$, for each $l=1,\dots,p$
\end{lemma}
\begin{proof}
The proof is by induction on $p\in\N^{\ge 2}$. For $p=2$, the desired result follows from Lemma~\ref{lem:p.equal.2}. 
Next, assume that Lemma \ref{lem:sparse.posi.form.general.case} holds for $p=\bar p-1$ and let us prove that it is also true for $p=\bar p$. 
By applying Lemma~\ref{lem:p.equal.2} with $I_1=\bigcup_{j = 1}^{{\bar p}-1} {{I_j}}  $, $I_2=I_{\bar p}$ and $f_1=\sum_{l = 1}^{{\bar p}-1} {{f_l}} $, $f_2=f_{\bar p}$, there exists a continuous rational function $\varphi_{\bar p}\in\R(x(\hat I_{\bar p}))$ defined by
\[\varphi_{\bar p}(y) := \frac{q_{\bar p}(y)}{\|y\|_2^{2k_{\bar p}}}\,,\, \forall y\in \R^{\hat n_{\bar p}}\,,\]
where $q_{\bar p}\in \R[x(\hat  I_{\bar p})]$ is homogeneous of degree $2(d+k_{\bar p})$ for some $k_{\bar p}\in\N$ (only depending on $d$, $\varepsilon$ and $f_1+\dots+f_{{\bar p}-1}$) such that 
\[f= h_{\{1,\dots,{\bar p}-1\}}+h_{\bar p}\,,\]
 where $h_{\{1,\dots,{\bar p}-1\}}:=f_1+\dots+f_{{\bar p}-1}-\varphi_{\bar p} \in\R\left(x\left(\bigcup_{j = 1}^{{\bar p}-1} {{I_j}}\right)\right)$ and $h_{\bar p}:=f_{\bar p}+\varphi_{\bar p}\in\R(x(I_{\bar p}))$ are continuous rational positive definite forms of degree $2d$.
By the RIP, there exists $s_{\bar p}\in\{2,\dots,{\bar p}-1\}$ such that $\hat I_{\bar p}\subset I_{s_{\bar p}}$, so $\varphi_{\bar p}\in\R(x(I_{s_{\bar p}}))$. 
Then $h_{\{1,\dots,{\bar p}-1\}}=\sum_{j = 1}^{{\bar p}-1} {(f_j-\delta_{j,s_{\bar p}} \varphi_{\bar p})}$ satisfies $f_j-\delta_{j,s_{\bar p}} \varphi_{\bar p}\in\R(x(I_j))$, $j=1,\dots,{\bar p}-1$. 
From this and by the induction hypothesis, there exist continuous rational functions $\varphi_l\in\R(x(\hat  I_l))$, $l=2,\dots,\bar p -1$, defined by
\[\varphi_l(y)= \frac{q_l(y)}{\|y\|_2^{2k_l}}\,,\, \forall y\in \R^{\hat n_l}\,,\,l=2,\dots,{{\bar p}-1}\,,\]
with $q_l\in \R[x(\hat I_l)]$ being homogeneous of degree $2(d+k_l)$ for some $k_l\in\N$, $l=2,\dots,{{\bar p}-1}$, such that 
\[ h_{\{1,\dots,{\bar p}-1\}}=\sum\limits_{l = 1}^{{\bar p}-1} h_l\,,\]
 where for $l=1,\dots,{{\bar p}-1}$,
\[h_l:=(f_l-\delta_{l,s_{\bar p}} \varphi_{\bar p})+\varphi_l-\sum\limits_{j = l+1}^{{{\bar p}-1}} \delta_{l,s_j}\varphi_j =f_l+\varphi_l-\sum\limits_{j = l+1}^{{{\bar p}}} \delta_{l,s_j}\varphi_j\in\R(x(I_l))\,,\] 
is a continuous rational positive definite form of degree $2d$.
Then $f=h_{\{1,\dots,{\bar p}-1\}}+h_{\bar p}=\sum_{l = 1}^{{\bar p}} h_l$, yielding the conclusion.
\end{proof}
The following result shows that one may  write a sparse rational positive definite form as a rational SOS with uniform denominator. 
%
\begin{lemma}\label{lem:sum.of.fractional}
Let ${I} = \bigcup_{l = 1}^p {{I_l}} $ and $d\in\N^{>0}$.
Let $f\in\R(x)$ be a rational positive definite form of degree $2d$ such that  $f=\sum_{l=1}^p \frac{q_l}{\|x(I_l)\|_2^{2k_l}}$, where $q_l\in \R[x( I_l)]$ is homogeneous of degree $2(d+k_l)$ for some $k_l\in\N$, $l=1,\dots,p$. 
Then there exists $\sigma\in \Sigma[x]_{d+k(p+1)}$ for some $k\in\N$ such that
\begin{equation}
    \label{lem3.6-1}
f=\frac{\sigma}{{\|x\|_2^{2k}}\prod_{l=1}^p {\|x(I_l)\|_2^{2k}}}\,.
\end{equation}
\end{lemma}
\begin{proof} 
Denote by $A$ the $\R$-algebra finitely generated by polynomials $x_j$, $j=1,\dots,n$, and rational functions $\frac{x(I_l)^\alpha}{\|x(I_l)\|_2^{2k_l}}$, $\alpha\in \N^{n_l}$ such that $|\alpha|=2(d+k_l)$, $l=1,\dots,p$. 
Let $C(\R^n)$ be the space of all continuous functions on $\R^n$. By Lemma \ref{lem:continuity.homogeneous},
the function $\frac{x(I_l)^\alpha}{\|x(I_l)\|_2^{2k_l}}$ is continuous 
for each $l=1,\dots,p$, and $\alpha\in \N^{n_l}$ with $|\alpha|=2(d+k_l)$.
Then $A$ is a commutative ring and $\R[x]\subset A\subset \R(x)\cap C(\R^n)$. 

Denote by $T$ the preordering generated by $\pm (1-\|x\|_2^2)$, i.e., $T$ consists of all elements of the form $\sigma+ (1-\|x\|_2^2) \psi$, for $\sigma\in \Sigma A^2$ and $\psi\in A$. 
Then $A$ is a preordered ring with fixed preordering $T$.  

We first prove that $\S^{n-1}=\{x\in\R^n\,:\,h(x)\ge 0\,,\,\forall h\in T\}$. Obviously $\S^{n-1}\subseteq \{x\in\R^n\,:\,h(x)\ge 0\,,\,\forall h\in T\}$. 
For the other inclusion, assume by contradiction that there exists $a\in \R^n\backslash \S^{n-1}$ such that $h(a)\ge 0$ for all $ h\in T$. Then $1-\|a\|_2^2\ne 0$. By selecting $h:=-(1-\|a\|_2^2)(1-\|x\|_2^2)\in T$, one obtains the contradiction
$0\le h(a)=-(1-\|a\|_2^2)^2<0$.

Next, notice that $\Sper_T A$ is a  Hausdorff space and contains all mappings $\hat a: A\to \R$, $h\mapsto h(a)$ for $a\in \S^{n-1}$ (see \cite{marshall2002general}). 
Here $\hat a$ is well-defined by the continuity of each element in $A$. 
In addition, since $\S^{n-1}$ is compact, $(\hat a)_{a\in \S^{n-1}}$ is dense in $\Sper_T A$ in the topology induced by the 
sup-norm, i.e., for each $r>0$ and for each $\varphi\in \Sper_T A$ there exists $a\in \S^{n-1}$ such that $\sup_{h\in A}|h(a)-\varphi(h)|=\sup_{h\in A}|(\hat a -\varphi)(h)|\le r$ (see \if{\cite[Example 1]{jacobi2001representation},}\fi \cite[Section 2]{marshall2001extending} and \cite[Section 2]{berr2001positive}).

Let $H(A)$ (resp. $H'(A)$) be the ring of geometrically (resp. arithmetically) bounded elements in $A$.
Since $(\hat a)_{a\in\S^{n-1}}$ is dense in $\Sper_T A$, 
\[H(A)=\{h\in A:\, h \text{ is bounded on }\S^{n-1}\}=A\,.\]
The latter equality is due to  the compactness of $\S^{n-1}$ and the inclusion $A\subset C(\R^n)$.
Combining this together with  \eqref{eq:geometrically.implies.arithmetically}, one obtains $A=H'(A)$. 

Next we claim that $f>0$ on $\Sper_T A$. Indeed $f\ge \varepsilon\|x\|_2^{2d}$ on $\R^n$ for some $\varepsilon>0$, because $f$ is a rational positive definite form of degree $2d$. Therefore $f\ge \varepsilon$ on $\S^{n-1}$. Let $\varphi\in \Sper_T A$ be fixed, arbitrary. By denseness of $\S^{n-1}$ in $\Sper_T A$, there exists $a\in \S^{n-1}$ such that $|f(a)-\varphi(f)|\le \frac \varepsilon 2$. Thus, $\varphi(f)=f(a)-(f(a)-\varphi(f))\ge \varepsilon- \frac \varepsilon 2=\frac \varepsilon 2>0$, and the result follows.

Next, since $f\in A$ and $f>0$ on $\Sper_T A$, then by Lemma \ref{lem:positive.in.preordering} $f\in T$. 
Therefore $f=\sigma+ (1-\|x\|_2^2) \psi$ for some $\sigma\in \Sigma A^2$ and  $\psi\in A$. 
By replacing $x$ by $\frac{x}{\|x\|_2}$ and noting that $f$ is homogeneous of degree $2d$, $\|x\|_2^{-2d}f=\sigma(\frac{x}{\|x\|_2})$. 
By multiplying both sides with ${\|x\|_2^{2(k+d)}}\prod_{l=1}^p {\|x(I_l)\|_2^{2k}}$ for some  large enough $k$, there exist $r\in\N$ and $h_j,v_j\in\R[x]$, $j=1,\dots,r$, such that
\[\begin{array}{rl}
\left({\|x\|_2^{2k}}\prod\limits_{l=1}^p {\|x(I_l)\|_2^{2k}}\right)f&=\left({\|x\|_2^{2(k+d)}}\prod\limits_{l=1}^p {\|x(I_l)\|_2^{2k}}\right)\sigma\left(\frac{x}{\|x\|_2}\right)\\
&=\sum\limits_{j=1}^r {(h_j+v_j\|x\|_2)^2}=\sum\limits_{j=1}^r {(h_j^2+v_j^2\|x\|_2^2)}+2\|x\|_2 \sum\limits_{j=1}^r {h_jv_j}\,.
\end{array}\]
Recall that $f=\sum_{l=1}^p \frac{q_l}{\|x(I_l)\|_2^{2k_l}}$. Therefore 
assume that $k$ is large enough to ensure that $\left({\|x\|_2^{2k}}\prod_{l=1}^p {\|x(I_l)\|_2^{2k}}\right)f$ is a polynomial. Then $\sum_{j=1}^r {(h_j^2+v_j^2\,\|x\|_2^2)}+2\|x\|_2 \sum_{j=1}^r {h_jv_j}$ must be a polynomial. However since $\|x\|_2$ is not a polynomial, then necessarily $\sum_{j=1}^r {h_jv_j}=0$. Hence,
\[\left({\|x\|_2^{2k}}\prod\limits_{l=1}^p {\|x(I_l)\|_2^{2k}}\right)f=\sum\limits_{j=1}^r {(h_j^2+v_j^2\|x\|_2^2)}\,,\]
which yields \eqref{lem3.6-1}.
\end{proof}
\begin{remark}
Observe that Reznick's Positivstellensatz is a particular case of Lemma \ref{lem:sum.of.fractional} with $p=1$. 
Our proof is similar to the one of \cite[Theorem 3.7]{schabert2019uniform}, which addresses the case $p=1$.
\end{remark}
\section{Proofs}
\label{sec:repre.theor}
\subsection{Proof of Theorem \ref{theo:rep.sum.pd}}
\label{proof:rep.sum.pd}
\begin{proof}
One has $f=\sum_{l=1}^p f_l$ with $f_l:=\frac{p_l}{\|x(\tilde I_l)\|^{2k_l}}$, $l=1,\dots,p$. 
By Lemma \ref{lem:continuity.homogeneous}, the  function $f_l\in\R(x(I_l))$ is continuous and homogeneous of degree $2d$, for each $l=1,\dots,p$.  
By applying Lemma \ref{lem:sparse.posi.form.general.case},
there exist continuous functions $\varphi_l\in\R(x(\hat I_l))$, $l=2,\dots,p$, defined by
\[\varphi_l(y)= \frac{q_l(y)}{\|y\|_2^{2k_l}}\,,\, \forall y\in \R^{\hat n_l}\,,\,l=2,\dots,p\,,\]
where $q_l\in \R[x(\hat I_l)]$ is homogeneous of degree $2(d+k_l)$ for some $k_l\in\N$, $l=2,\dots,p$, and one has
\[ f=\sum\limits_{l = 1}^p h_l\,,\]
 where each $h_l:=f_l+\varphi_l-\sum_{j = l+1}^{p} \delta_{l,s_j}\varphi_j \in\R(x(I_l))$,  $l=1,\dots,p$ (with $\varphi_1:=0$) is a  continuous rational positive definite form of degree $2d$.
Then, we apply Lemma \ref{lem:sum.of.fractional} with the notation $f \leftarrow h_l$, $I \leftarrow I_l$ and $I_l \leftarrow I_l \cup \{\hat I_j:\,s_j=l,\,j=l+1,\dots,p\}$. 
Therefore, there exist $\hat k_l\in\N$ and $\psi_l\in \Sigma[x( I_l)]_{d+\hat k_l(1+\deg(\Phi_l)/2)}$  such that
\[h_l=\frac{\psi_l}{\|x( I_l)\|_2^{2\hat k_l}\|x(\hat I_l)\|_2^{2(1-\delta_{l,1})\hat k_l}\prod\limits_{j = l+1}^{p}{ \|x(\hat I_j)\|_2^{2\hat k_l\delta_{l,s_j}}}} \,,\, l=1,\dots,p\,.\] 
Let $k:=\max\{\hat k_1,\dots,\hat k_p\}$ and define for all $l=1,\dots,p$
\[\sigma_l:=\psi_l\|x( I_l)\|_2^{2(k-\hat k_l)}{\|x(\hat I_l)\|_2^{2(1-\delta_{l,1})(k-\hat k_l)}\prod\limits_{j = l+1}^{p}{ \|x(\hat I_j)\|_2^{2(k-\hat k_l)\delta_{l,s_j}}}}\,.\]
Then $\sigma_l\in \Sigma[x( I_l)]_{d+k (1+\deg(\Phi_l)/2)}$ and \eqref{eq:homogeneous} follows, yielding the conclusion.
\end{proof}
\subsection{Proof of Corollary \ref{coro:rep.sum.positive}}
\label{proof:rep.sum.positive}
\begin{proof}
Let $\bar f(x,x_{n+1}):=x_{n+1}^{2d}f({x}/{x_{n+1}})$ be the degree-$2d$ homogenization of $f$. 
Set $\bar x:=(x,x_{n+1})$, $\bar I:=I\cup\{n+1\}$, $\bar I_l:=I_l\cup\{n+1\}$ and $\hat {\bar I}_l:=\hat I_l\cup\{n+1\}$, for all $l=1,\dots,p$. 
Then ${\bar I} = \bigcup_{l = 1}^p {{\bar I_l}} $ and 
\[\forall l \in \{ {1,\dots,p} \}\,,\,\hat {\bar I}_l={\bar I_{l}} \cap \left( {\bigcup\limits_{j = 1}^{l-1} {{\bar I_j}} } \right) \subset {\bar I_{s_l}}\quad\text{and}\quad
\bar f = \sum\limits_{l = 1}^p {{\bar f_l}} \,.\]
Since $f$ is nonnegative, $\bar f$ is also nonnegative. 

We first prove that $\bar f+\varepsilon \sum_{l = 1}^p {\|\bar x(\bar I_l)\|_2^{2d}}$ is a positive definite form. 
Let $y\in\R^{n+1}$ such that $\bar f(y)+\varepsilon \sum_{l = 1}^p {\|y(\bar I_l)\|_2^{2d}}=0$. 
By the nonnegativity of $\bar f$ and $\|\bar x(\bar I_l)\|_2^{2d}$, $l=1,\dots,p$,  
\[\bar f(y)= {\|y(\bar I_1)\|_2^{2d}}=\dots={\|y(\bar I_p)\|_2^{2d}}=0\,.\]
Hence $y(\bar I_l)=0$, $l=1,\dots,p$, and therefore since ${\bar I} = \bigcup_{l = 1}^p {{\bar I_l}} $, $y=0$.

 By Theorem \ref{theo:rep.sum.pd}, there exist $k\in\N$ and $ { \psi}_l\in \Sigma[\bar x( \bar I_l)]_{d+k\omega_l}$, $l=1,\dots,p$,  such that 
\begin{equation}\label{eq:homogeneous.formula}
\bar f+\varepsilon \sum\limits_{l = 1}^p {\|\bar x(\bar I_l)\|_2^{2d}}=\sum\limits_{l = 1}^p {\frac{{ \psi}_l }{\|\bar x( {\bar I}_l)\|_2^{2k} \tilde \Phi_l^{k}}}\,,
\end{equation}
 where $\tilde  \Phi_l:={\|\bar x(\hat {\bar I}_l)\|_2^{2(1-\delta_{l,1})}\prod_{j = l+1}^{p}{ \|\bar x(\hat {\bar I}_j)\|_2^{2\delta_{l,s_j}}}}\in \R[\bar x( \bar I_l)]$ and $\omega_l=\deg(\tilde  \Phi_l)+1$, $l=1,\dots,p$.
Letting $x_{n+1}:=1$ in \eqref{eq:homogeneous.formula} yields
\begin{equation*}
f+\varepsilon \sum\limits_{l = 1}^p {\theta_l^{d}}=\sum\limits_{l = 1}^p {\frac{\sigma_l }{\theta_l^kD_l^k}}\,,
\end{equation*}
 with $D_l=\tilde \Phi_l(x,1)\in \R[x( I_l)]$ and $\sigma_l:= {\psi}_l (x,1)\in \Sigma[x( I_l)]_{d+k\omega_l}$.
Hence the conclusion follows since $\Theta_l=\theta_lD_l$, $l=1,\dots,p$.
\end{proof}
\subsection{Proof of Corollary \ref{coro:rep.sum.positive.con}}
\label{proof:rep.sum.positive.con}
\begin{proof}
Recall that $u_j=\lceil \deg(g_j)/2\rceil$, for all $j=1,\dots,m$.
Define $\lambda_j:=(\|g_j\|_1+1)^{-1}$, for all $j=1,\dots,m$. 
We claim that 
\begin{equation}\label{eq:bound.by.1}
\left|\frac{\lambda_jg_j}{\theta_l^{u_j}}\right|\le \frac{\|g_j\|_1}{\|g_j\|_1+1}<1\,,\,j\in J_l\,,\, l=1,\dots,p\,.
\end{equation}
Indeed $|x^{\alpha}|\le \theta_l^{u_j}$ for all $\alpha\in\N^{I_l}_{2u_j}$, $j\in J_l$, $l=1,\dots,p$, which implies
\[\left|\frac{g_j}{\theta_l^{u_j}}\right| = \left|\frac{\sum_{\alpha\in\N^n_{2u_j}}{g_{j,\alpha}x^{\alpha}}}{\theta_l^{u_j}}\right| \le \sum\limits_{\alpha\in\N^n_{2u_j}} {|g_{j,\alpha}|\frac{|x^{\alpha}|}{\theta^{u_j}_l}}\le \|g_j\|_1\,.\]
For each $k\in\N$ introduce
\begin{equation}\label{eq:formula.f.k}
Q_k:=f+\frac{\varepsilon} 2 \sum\limits_{l = 1}^p {\theta_l^{d}} - \sum\limits_{l = 1}^p \sum\limits_{j\in J_l} 
{{\left(1-\frac{\lambda_jg_j}{\theta_l^{u_j}}\right)^{2k^2}}\left(\frac{\lambda_jg_j}{\theta_l^{u_j}}\right)^{2k+1}} \,.
\end{equation}
Let us show that $Q_k$ is nonnegative for $k$ sufficiently large. 
Since $2d> \deg (f)$, $f+\frac{\varepsilon} 2 \sum_{l = 1}^p {\theta_l^{d}}$ is coercive, i.e., $\lim_{\|x\|_2\to \infty} \left(f+\frac{\varepsilon} 2 \sum_{l = 1}^p {\theta_l^{d}}\right)=\infty$. 
Thus, there exists $M>0$ such that $f+\frac{\varepsilon} 2 \sum_{l = 1}^p {\theta_l^{d}}\ge 1$ on $B(0,M)^c$, where $B(0,M)$ stands for the open ball of radius M centered at the origin.

\noindent (I) Let $y\in B(0,M)^c$ be fixed and  $\Lambda:=\{j\in J:\,g_j(y)\ge 0\}$. 
For each $k \in \N$
\[\begin{array}{rl}
Q_k(y)\ge &1 - \sum\limits_{l = 1}^p \sum\limits_{j\in J_l\cap \Lambda}{{\left(1-\frac{\lambda_jg_j(y)}{\theta_l(y)^{u_j}}\right)^{2k^2}}\left(\frac{\lambda_jg_j(y)}{\theta_l(y)^{u_j}}\right)^{2k+1}}\\
&- \sum\limits_{l = 1}^p \sum\limits_{j\in J_l\backslash \Lambda}{{\left(1-\frac{\lambda_jg_j(y)}{\theta_l(y)^{u_j}}\right)^{2k^2}}\left(\frac{\lambda_jg_j(y)}{\theta_l(y)^{u_j}}\right)^{2k+1}}\\
\ge & 1 - \sum\limits_{l = 1}^p \sum\limits_{j\in J_l\cap \Lambda}{{\left(1-\frac{\lambda_jg_j(y)}{\theta_l(y)^{u_j}}\right)^{2k^2}}\left(\frac{\lambda_jg_j(y)}{\theta_l(y)^{u_j}}\right)^{2k+1}}\\
\ge & 1 - \sum\limits_{l = 1}^p \sum\limits_{j\in J_l\cap \Lambda}{\left(\frac{\|g_j\|_1}{\|g_j\|_1+1}\right)^{2k+1}}\qquad \text{(by \eqref{eq:bound.by.1})}\\
\ge & 1 - \sum\limits_{l = 1}^p \sum\limits_{j\in J_l}{\left(\frac{\|g_j\|_1}{\|g_j\|_1+1}\right)^{2k+1}}
\,.
\end{array}\]
Since $y$ was arbitrary in $B(0,M)^c$, 
\[\inf\{Q_k(x):\,{x\in B(0,M)^c}\}\ge 1 - \sum\limits_{l = 1}^p \sum\limits_{j\in J_l}{\left(\frac{\|g_j\|_1}{\|g_j\|_1+1}\right)^{2k+1}}\to 1\text{ as }k\to \infty\,.\]
Thus, $Q_k$ is nonnegative on $ B(0,M)^c$ for some large enough $k$. 

\noindent (II) Then, note that $\lim_{k\to\infty}-(1-a)^{2k^2}a^{2k+1}=0$ for all $a\in(0,1)$ and $\lim_{k\to\infty}(1+a)^{2k^2}a^{2k+1}=\infty$ for all $a\in(0,1)$.
By using \eqref{eq:bound.by.1}, each term $-{{\left(1-\frac{\lambda_jg_j}{\theta_l^{u_j}}\right)^{2k^2}}\left(\frac{\lambda_jg_j}{\theta_l^{u_j}}\right)^{2k+1}}$ involved in \eqref{eq:formula.f.k} can be written either as $-(1-a)^{2k^2}a^{2k+1}$ when $g_j(x) > 0$ or as $(1+a)^{2k^2}a^{2k+1}$ when $g_j(x) < 0$ for some $a \in (0, 1)$.
Therefore,  $Q_k\to \infty$ pointwise on $\overline{B(0,M)}\backslash S(g)$ and $Q_k\to f+\frac{\varepsilon} 2 \sum_{l = 1}^p {\theta_l^{d}}$ pointwise on $\overline{B(0,M)}\cap S(g)$. 
By compactness of $\overline{B(0,M)}$ and  positivity of $f+\frac{\varepsilon} 2 \sum_{l = 1}^p {\theta_l^{d}}$ on $S(g)$, $Q_k$ is nonnegative on $\overline{B(0,M)}$ for large enough $k$. 

\noindent (III) Let $K\in\N$ be fixed such that $Q_K$ is nonnegative. 
Define $r_j:=(2K^2+2K+1)u_j$ and $w_{j,l}:=(\theta_l^{u_j}-\lambda_jg_j)^{2K^2}(\lambda_jg_j)^{2K+1}$, so that $w_{j,l}\in \R[x(I_l)]_{2r_j}$, $j\in J_l$, $l=1,\dots,p$, and 
\[Q_K\,=\,f+\frac{\varepsilon} 2 \sum\limits_{l = 1}^p {\theta_l^{d}} - \sum\limits_{l = 1}^p \sum\limits_{j\in J_l}{\frac{w_{j,l}}{\theta_l^{r_j}}}\,.\]
With every $h\in\R(x)$ associate its degree-$2d$ homogenization $\bar h(x,x_{n+1}):=x_{n+1}^{2d}h({x}/{x_{n+1}})$. 
Then with same notation $\bar x$, $\bar I$, $\bar I_l$ and $\hat {\bar I}_l$ as in the proof of Corollary~\ref{coro:rep.sum.positive}:
\begin{equation} \label{eq:homogeneous.f.K}
\bar Q_K=\bar f+\frac{\varepsilon} 2 \sum\limits_{l = 1}^p {\|\bar x(\bar I_l)\|_2^{2d}} - \sum\limits_{l = 1}^p \sum\limits_{j\in J_l}{\frac{x_{n+1}^{2r_j}\bar w_{j,l}}{\|\bar x(\bar I_l)\|_2^{2r_j}}}\,.
\end{equation}
Equivalently:
\[\bar Q_K+\frac{\varepsilon} 2 \sum\limits_{l = 1}^p {\|\bar x(\bar I_l)\|_2^{2d}} = \sum\limits_{l = 1}^p {\left(\bar f_l+{\varepsilon}\|\bar x(\bar I_l)\|_2^{2d}-\sum\limits_{j\in J_l}{\frac{x_{n+1}^{2r_j}\bar w_{j,l}}{\|\bar x(\bar I_l)\|_2^{2r_j}}}\right)}=\sum\limits_{l = 1}^p F_l \,,\]
where $F_l:={\bar f_l+{\varepsilon}\|\bar x(\bar I_l)\|_2^{2d}-\sum_{j\in J_l}{\frac{x_{n+1}^{2r_j}\bar w_{j,l}}{\|\bar x(\bar I_l)\|_2^{2r_j}}}}\in\R(\bar x(\bar I_l))$ is homogeneous of degree $2d$, $l=1,\dots,p$.
In addition, $\bar Q_K$ is nonnegative by nonnegativity of $Q_K$.
Then there exists $\bar \varepsilon > 0$ such that
$\bar Q_K+\frac{\varepsilon} 2 \sum_{l = 1}^p {\|\bar x(\bar I_l)\|_2^{2d}} \ge \bar \varepsilon \|\bar x\|_2^{2d}$. 
Indeed, for $d=1$, it is trivial. For $d\ge 2$, let $d^*$ be such that $\frac{1}{d}+\frac{1}{d^*}=1$  and let us use H\"older's inequality as follows:
\[\left(\sum\limits_{l = 1}^p {\|\bar x(\bar I_l)\|_2^{2d}}\right)^{1/d}\left(\sum\limits_{l = 1}^p 1^{d^*}\right)^{1/d^*}\ge \sum\limits_{l = 1}^p {\|\bar x(\bar I_l)\|_2^{2}}\ge \|\bar x\|_2^{2}\,,\]
which implies the desired result for $\bar \varepsilon = \frac{\varepsilon}{2} p^{-d/d^*}$. 
Therefore $\bar Q_K+\frac{\varepsilon} 2 \sum_{l = 1}^p {\|\bar x(\bar I_l)\|_2^{2d}}$ is a positive definite form of degree $2d$. By Theorem \ref{theo:rep.sum.pd}, there exist $k\in\N$ and  $\hat \psi_l\in \Sigma[\bar x( \bar I_l)]_{d+k\omega_l}$, $l=1,\dots,p$,  such that 
\begin{equation*}
 \bar Q_K+\frac{\varepsilon} 2 \sum\limits_{l = 1}^p {\|\bar x(\bar I_l)\|_2^{2d}}=\sum\limits_{l = 1}^p {\frac{\hat \psi_l}{\|\bar x(\bar I_l)\|_2^{2k}\hat D_l^k }}\,.
\end{equation*}
 where $\hat D_l:={\|\bar x(\hat{\bar I}_l)\|_2^{2(1-\delta_{l,1})}\prod_{j = l+1}^{p}{ \|\bar x(\hat{\bar I}_j)\|_2^{2\delta_{l,s_j}}}}\in \R[\bar x( \bar I_l)]$, $l=1,\dots,p$. 
From this and \eqref{eq:homogeneous.f.K},
\[\bar f+{\varepsilon} \sum\limits_{l = 1}^p {\|\bar x(\bar I_l)\|_2^{2d}} = \sum\limits_{l = 1}^p \sum\limits_{j\in J_l}{\frac{x_{n+1}^{2r_j}\bar w_{j,l}}{\|\bar x(\bar I_l)\|_2^{2r_j}}}+\sum\limits_{l = 1}^p {\frac{\hat \psi_l}{\|\bar x(\bar I_l)\|_2^{2k}\hat D_l^k }}\,.\]
Letting $x_{n+1}:=1$ yields
\begin{equation*}
f+\varepsilon \sum\limits_{l = 1}^p {\theta_l^{d}}=\sum\limits_{l = 1}^p \sum\limits_{j\in J_l}{\frac{w_{j,l} }{\theta_l^{r_j}}}+\sum\limits_{l = 1}^p {\frac{\sigma_{0,l} }{\theta_l^{k}D_l^{k}}}\,,
\end{equation*}
 with $D_l=\hat D_l(x,1)\in \R[x( I_l)]$ and $\sigma_{0,l}:=\hat { \psi}_l (x,1)\in \Sigma[x( I_l)]_{d+k\omega_l}$. 
For $j\in J_l$, $l=1,\dots,p$, by setting $\sigma_{j,l}:=(\theta_l^{u_j}-\lambda_jg_j)^{2K^2}(\lambda_jg_j)^{2K}$, $\sigma_{j,l}\in \Sigma[x(I_l)]_{r_j-u_j}$ and $w_{j,l}=\sigma_{j,l}g_j$. 
By setting $\tilde k:=\max\{k,\,r_j:\,j\in J\}$,
\[ \begin{array}{rl}
f+\varepsilon \sum\limits_{l = 1}^p {\theta_l^{d}}&=\displaystyle\sum\limits_{l = 1}^p \sum\limits_{j\in J_l}{\frac{\theta_l^{\tilde k-r_j}D_l^{\tilde k}\sigma_{j,l}g_j }{\theta_l^{\tilde k}D_l^{\tilde k}}}+\sum\limits_{l = 1}^p {\frac{D_l^{\tilde k-k}\theta_l^{\tilde k-k}\sigma_{0,l} }{\theta_l^{\tilde k} D_l^{\tilde k}}}\\
&=\displaystyle\sum\limits_{l = 1}^p \sum\limits_{j\in J_l}{\frac{\tilde \sigma_{j,l}g_j }{\theta_l^{\tilde k}D_l^{\tilde k}}}+\sum\limits_{l = 1}^p {\frac{\tilde \sigma_{0,l} }{\theta_l^{\tilde k}D_l^{\tilde k}}}\,,
\end{array}\]
with $\tilde \sigma_{0,l}:={D_l^{\tilde k-k}\theta_l^{\tilde k-k}\sigma_{0,l} }\in \Sigma[x( I_l)]_{d+\tilde k\omega_l}$ and 
\[\tilde \sigma_{j,l}:=\theta_l^{\tilde k-r_j}D_l^{\tilde k}\sigma_{j,l}\in \Sigma[x(I_l)]_{\tilde k\omega_l-u_j}\subset \Sigma[x(I_l)]_{d+\tilde k\omega_l-u_j}\,,\, j\in J_l\,,\, l=1,\dots,p\,.\]
 Thus,
\begin{equation*}
f+\varepsilon \sum\limits_{l = 1}^p {\theta_l^{d}}=\sum\limits_{l = 1}^p {\frac{{{\tilde \sigma_{0,l} }}+ \sum_{j\in J_l}{{\tilde \sigma_{j,l} {g_j} }}}{\theta_l^{\tilde k}D_l^{\tilde k}} }\,.
\end{equation*}
Hence, the conclusion follows since $\Theta_l=\theta_lD_l$, $l=1,\dots,p$.
\end{proof}
\section{Conclusion}
In this paper, we have provided: 

- a sparse version for both Reznick's Positivstellensatz (resp. Putinar-Vasilescu's Positivstellensatz) for positive definite forms (resp. nonnegative polynomials).

- a sparse version of Putinar-Vasilescu's Positivstellensatz for polynomials that are nonnegative on a
possibly non-compact basic semialgebraic set.

All these certificates involve sums of squares of rational functions with uniform denominators and a topic of further research is how to exploit such positivity certificates  in polynomial optimization on non-compact basic semialgebraic sets.
\paragraph{\textbf{Acknowledgements}.} 
%
The first author was supported by the MESRI funding from EDMITT.
The second author was supported by the FMJH Program PGMO (EPICS project) and  EDF, Thales, Orange et Criteo, as well as from the Tremplin ERC Stg Grant ANR-18-ERC2-0004-01 (T-COPS project).
This work has benefited from the Tremplin ERC Stg Grant ANR-18-ERC2-0004-01 (T-COPS project), the European Union's Horizon 2020 research and innovation programme under the Marie Sklodowska-Curie Actions, grant agreement 813211 (POEMA) as well as from the AI Interdisciplinary Institute ANITI funding, through the French ``Investing for the Future PIA3'' program under the Grant agreement n$^{\circ}$ANR-19-PI3A-0004.
The third author was supported by the European Research Council (ERC) under the European's Union Horizon 2020 research and innovation program (grant agreement 666981 TAMING).
\appendix
\section{Appendix}
\label{sec:Appendix}
\tiny
\begin{align*}
\sigma_1=&
2(x_1 x_2^2 - 1/2 x_1 x_2 + \frac{9850453248969845}{54043195528445952} x_1 - \frac{5320372455790409}{18014398509481984} x_2)^2 \\
&+ \frac{23926543956308875}{13510798882111488}(x_1 x_2 - \frac{11060480396851749}{47853087912617750} x_2^2 - \frac{44192742279476107}{95706175825235500} x_1 \\
&+ \frac{644586703866657}{4785308791261775} x_2)^2 + \frac{562653141360639470744536704411533}{431022297783585770305704951808000} (x_2^2 \\
&- \frac{623163665459468362849659231125131}{1125306282721278941489073408823066} x_1 + \frac{368582633567496760334523605556710}{562653141360639470744536704411533} x_2)^2 \\
&+\frac{639450287086164028376019645640404206647864623032556298769898592967}{4108306132028813363573605000030024459546025044241423314110728110080} (x_1 \\
&- \frac{1267974608530305169381220379238785299625590449899169509120071486121}{639450287086164028376019645640404206647864623032556298769898592967} x_2)^2 \\
&+ \frac{543123373122508979151195069765687696047832930661622412978516470349}{25578011483446561135040785825616168265914584921302251950795943718680} x_2^2\,,
\end{align*}
and
\begin{align*}
\sigma_2=&
 142(x_2^2 x_3 + \frac{9}{142}x_2^2 - \frac{27}{142}x_2 x_3 + \frac{16330772824785797}{2558044588346441728}x_2 + \frac{16516005877459571}{1918533441259831296}x_3)^2 \\
&+ \frac{82370346211711367}{79938893385826304}(x_2^2 - \frac{249757745830951395}{1317925539387381872}x_2 x_3 - \frac{2229916595909069297}{2635851078774763744} x_2 \\
&- \frac{444311853617472399}{658962769693690936}x_3)^2 + \frac{9998861038823771954693138032308042163}{71224907617044875374639550136582144}(x_2 x_3 \\
&+ \frac{1371002666962875555152486340866448945}{19997722077647543909386276064616084326}x_2 \\
&- \frac{1915895198624799259160186828281383330}{9998861038823771954693138032308042163}x_3)^2 \\
&+ \frac{1314993284459107346080199880406085742469204627868504341150915245917907}{6489631845101369284182723987666581388983892796141390004625948809363456}(x_2 \\
&+ \frac{4839609578327155219592253159785080664809790873603364761086184116575700}{3944979853377322038240599641218257227407613883605513023452745737753721}x_3)^2 \\
&+ \frac{206106598111713320412255700324878936837424344667666864964748474281899081711403351953}{3331239337060166527019274109936287427427959404917141471434825520589496162184413052928}x_3^2\,.
\end{align*}
\footnotesize
\bibliographystyle{abbrv}

\end{document}